\documentclass[a4]{article}

\usepackage{amsmath,amsthm,amsfonts,amssymb,amscd,epsf,epsfig,psfrag,enumerate, tikz,bbold}
\usepackage[margin=3.5cm]{geometry}

\def\conv{\operatorname{conv}}
\def\dom{\operatorname{dom}}
\def\norm{\operatorname{||}}

\def\intt{\operatorname{int}}

\def\bd{\operatorname{bd}}
\def\supp{\operatorname{supp}}
\DeclareMathOperator*{\argmin}{argmin}

\def\R{\mathbb{R}}\def\Z{\mathbb{Z}}\def\Q{\mathbb{Q}}

\def\T{\mathsf{T}}
\def\l{\,\mid\,}

\newtheorem{theorem}{Theorem}

\newtheorem{definition}[theorem]{Definition}

\newtheorem{corollary}[theorem]{Corollary}

\newtheorem{example}[theorem]{Example}

\title{\bf Duality for Mixed-Integer Convex Minimization}

\author{Michel Baes, Timm Oertel and Robert Weismantel	}
 
\begin{document}

\date{December 8, 2014}

\maketitle

\begin{abstract}
We extend in two ways the standard Karush-Kuhn-Tucker optimality conditions to problems with a convex objective, convex functional constraints, and the extra requirement that some of the variables must be integral. While the standard Karush-Kuhn-Tucker conditions involve separating hyperplanes, our extension is based on lattice-free polyhedra. Our optimality conditions allow us to define an exact dual of our original mixed-integer convex problem. 
\end{abstract}

\section{Introduction}

Several attempts have been made in the past to define formally a dual of a linear integer  or mixed-integer programming problem. Let us first mention some important developments in this direction. 

One idea to define a dual program associated with a binary linear integer programming problem is to encode the given
$0/1$-problem in form of a linear program in an extended space, so that the new variables correspond to linearizations of products of original variables. The variables of the dual of the resulting linear optimization program can be reinterpreted in terms of the original variables. This concept of duality has its origins in the work of \cite{SheraliAdams90,SheraliAdams99} and \cite{LovaszSchrijver91} and is closely connected with the earlier work of \cite{Balas75a,Balas85b} on disjunctive optimization. 
It also provides us an interesting link to the theory of polynomial optimization  including
duality results associated with hierarchies of semidefinite programming problems, see \cite{Lasserre13}.
For a nice survey treating the relationships and differences between several relaxations of this kind we refer to 
\cite{Laurent01}.

A second important development in integer optimization is based on the connections between valid inequalities and subadditive functions. This leads to a formalism that allows us to establish
a subadditive dual of a general mixed integer linear optimization problem, see \cite{Johnson73,Jeroslow78,Jeroslow79,Jeroslow1979,Johnson79}
and \cite{NemhauserWolsey88} for a treatment of the subject and further references. Recently, a strong subadditive
dual for conic mixed integer optimization has been constructed in \cite{RW_MDV2012}.

There are several other special cases for which the dual of a mixed integer optimization problem has been derived. One such example is based on the theory of discrete convexity established in \cite{RW_M2003}. Here, an explicit dual is constructed for L-convex and M-convex functions.

A third general approach to develop duality in several subfields of optimization is based on the Lagrangian relaxation method.
The latter method is broadly applicable and -- among others -- leads to a formalism of duality in convex optimization.  
  The connection between the Lagrangian dual and linear
  relaxations of linear integer optimization problems
has its origins in  \cite{HeldKarp70}.  This paper developed a combinatorial version of a Lagrangian relaxation in form of 1-trees for the  traveling salesman problem. Since then there is a large number 
of applications  using this relaxation technique for integer optimization problems, see e.g. Chapter II.3 in \cite{NemhauserWolsey88}. 

Our point of departure is the strong duality theorem for convex optimization based on the Lagrangian relaxation method. We will show that optimality certificates and duality in convex optimization have a very natural mixed-integer analogue. 
A duality theory in Euclidean space follows from a precise interplay between points -- that are viewed as primal objects --  and level sets of linear functions, that is, closed half-spaces, interpreted as dual objects. It turns out that there is a similar interplay in the mixed-integer setting. Here, the primal objects are sets of points, whereas the dual objects are lattice-free open polyhedra.
Our motivation for studying optimality certificates and a mixed-integer convex dual comes from the important developments in convex optimization in the past decade. As a first step towards new mixed-integer convex algorithms, it seems natural to make an attempt of extending some of the basic convex optimization tools to the mixed-integer setting.

\section{Mixed-integer optimality certificates}

Let $f:\dom(f)\mapsto\R$ be a continuous convex function.
In order to simplify our exposition we may assume w.l.o.g. here that $\dom(f) = \R^n$. Assume that $f$ has a, not necessarily unique, minimizer $x^\star$. 
Then a necessary and sufficient certificate for $x^\star$ being a minimizer of $f$ is that
$0\in\partial f(x^\star)$, i.e. the zero-function is in the subdifferential of $f$ at $x^\star$.
Hence
$$x^\star=\argmin\limits_{x\in\R^n} f(x) \Longleftrightarrow 0\in\partial f(x^\star).$$
The question  emerges how to obtain a certificate that a point $x^\star\in\Z^n \times \R^d$ solves the corresponding mixed-integer convex problem 
\begin{equation}\label{eq:unconstrained_MICO}
x^\star=\argmin\limits_{x\in\Z^n \times \R^d} f(x)?
\end{equation}
Let us first explain the idea of our approach. By definition, $x^\star=\argmin\limits_{x\in\Z^n \times \R^d} f(x)$ if and only if 
\begin{equation}
\label{bbwe}
 \{ x\in\Z^n \times \R^d \mid  f(x) < f(x^\star)\} = \emptyset.
\end{equation}
The level set $\{ x\in\R^{n+d} \mid  f(x) \leq f(x^\star)\}$ is convex. If it is nonempty, then its projection to its first $n$ components, that is, to the subspace spanned by the integer variables is again a convex set.  Clearly,  $x^\star=\argmin\limits_{x\in\Z^n \times \R^d} f(x) $ if and only if 
$$ Q:=\{ {z} \in\R^{n} \mid \exists y \in \R^d, x= {(z,y)} \text{ and } f(x) < f(x^\star)\} \cap \Z^n = \emptyset.$$
From a theorem of Lovasz, inclusionwise maximal lattice-free convex sets are polyhedra \cite{Lovasz89}: we can restrict our attention to such polyhedra $P$ that contain the (convex) projection $Q$. From the theorem of Doignon \cite{Doignon1973}, it follows that a subset of at most $2^n$ inequalities in the description of $P$ are enough to prove that $\intt(P)  \cap \Z^n = \emptyset$.
It remains to show how to relate {these} $2^n$ inequalities to the function $f$. The following theorem, which is an immediate consequence of our Theorem \ref{mixedIntgerKKT} proved at the end of this section, clarifies this relationship, providing a necessary and sufficient condition for our original mixed-integer convex problem. Each of these $2^n$ inequalities is related to a mixed-integer point, the set of which constitutes our \emph{optimality certificate}. Condition (a) ensures that $x^\star$ is one of the points of the optimality certificate and is the best of them. Also, in view of Condition (c), every point $x$ in the certificate minimizes $f$ on its own fiber, that is, in the set $\{(x_1,\ldots,x_n)^T\}\times\R^d$. Finally, the subgradient of $f$ at each point of the certificate defines a half-space. The interior of their intersection plays the role of a polyhedron whose projection on the first $n$ components is the $P$ described above. Condition (b) ensures that this interior is lattice-free.

One direct implication of the result {below} is that it provides us with a certificate for the optimality of a given mixed-integer point. The  verification can be performed in polynomial time, provided that the number of integer variables is a constant.

\begin{theorem}\label{mixedIntgerunconstrained}
$x^\star=\argmin\limits_{x\in\Z^n \times \R^d} f(x)$  
if and only if  there exist $k\le 2^n$ points
$x_1=x^\star,x_2,\dots,x_k\in\Z^{n}\times\R^{d}$ and vectors $h_{i }\in\partial f(x_i)$
such that the following conditions hold:
\begin{itemize}
\item[(a)] $f(x_1)\le \ldots \le f(x_k)$.
\item[(b)] $\{x\in\R^{n+d} \l   h_{i}^\T (x-x_i) < 0 \text{ for all } i\}\cap(\Z^{n}\times\R^d)=\emptyset$.
\item[(c)] $h_{i} \in\R^n\times\{0\}^d$ for $i=1,\dots,k$.\end{itemize}
\end{theorem}

{As announced, we formulate and prove a version of the above theorem that takes possible convex functional constraints to problem \eqref{eq:unconstrained_MICO} into account.}

{Let} $g_1,\ldots,g_m:\dom(f)\mapsto\R$ be continuous convex functions. Again we may assume w.l.o.g. that $\dom(g_j) = \R^n$ for all $j$. By $g(x)$ we denote the vector of components $g_1(x), \ldots,g_m(x)$.  Let us first discuss the continous convex optimization problem  
\begin{equation}\label{eq:constrained_MICO}
x^\star=\argmin\limits_{x\in\R^n,\atop g(x)\le0} f(x).
\end{equation}
Assume that there exists a point {$y\in\R^n$} fulfilling the so-called Slater condition, that is, $g(y)<0$.
Under this assumption the Karush-Kuhn-Tucker (KKT) conditions (e.g. \cite{Karush39,KuhnTucker51}) 
provide necessary and  sufficient optimality conditions.
Namely, the point $x^\star$ such that $ g(x^\star)\le0$ attains the optimal continuous solution if and only if there exist $h_f\in\partial f(x^\star)$, $h_{g_i}\in\partial g_i(x^\star)$, for $i=1,\dots,m$ and
non-negative $\lambda_i$, $i=1,\dots,m$, such that
\begin{equation}\label{eq:KKT_continuous}
h_f+\sum_{i=1}^m\lambda_i h_{g_i} = 0 \text{ and } \lambda_i g_i(x^\star)=0 \; \forall i. 
\end{equation}
Note that in this representation it suffices  to consider only those $g_i(x^\star)$ that are active, i.e. $\lambda_i\neq 0$,{ and for which the corresponding $h_{g_i}$ are linearly independent}.

It is our intention to generalize these optimality conditions to the mixed-integer setting
\begin{equation}\label{chapter:CenterPoints::equ:mixedIntegerOpti}
x^\star=\argmin\limits_{x\in\Z^{n}\times\R^d,\atop g(x)\le0} f(x).
\end{equation}
We first generalize the Slater condition.

\begin{definition} 
We say that the constraints $g(x)\le0$ fulfill the  {\it mixed-integer Slater condition} if for every point ${(z,y)}\in\Z^n\times\R^d$ with $g({(z,y)})\le0$ there exists a ${y'}\in\R^d$ such that $g({(z,y')})<0$.
\end{definition}

Under the assumption of the mixed-integer Slater condition, we next formulate and justify mixed-integer optimality conditions.

As in the unconstrained situation, a certificate is given by a list of $k\leq 2^n$ mixed-integer points $x_1,x_2,\ldots,x_k$. Our substitute for the nonnegative multipliers $\lambda$ are $k$ nonnegative vectors $u_1,\ldots,u_k$ of size $m+1$. Condition (c) asserts that none of these vectors is null; their maximal sparsity is a consequence of Caratheodory's Theorem.
As in Theorem \ref{mixedIntgerunconstrained}, Condition (e) indicates that the points of the certificate are optimal in their own fiber. Condition (a) ensures that the mixed-integer optimum $x^\star$ is in the certificate, and is the best among all those that are in it. The additional complementarity conditions are inherited from the continuous KKT theorem quoted above. 
Among the certificate points $x_2,\ldots,x_k$, several might be infeasible for the primal problem. Let us define the set $I_i:=\{1\leq j\leq m\mid g_j(x_i) = \max_{1\leq \ell\leq m}g_{\ell}(x_i)\}$ for every $1\leq i\leq k$. Condition (b) addresses infeasible points $x_i$ in the certificate, that is, those for which $g_j(x_i)>0$ for any $j\in I_i$. It ensures that $u_{i,j}>0$ only when $j\notin I_i$. (Note that the complementarity condition for feasible points $x_i$'s in Condition (a) can be expressed identically). The lattice-freeness Condition (d) is the natural extension of Condition (b) in Theorem \ref{mixedIntgerunconstrained}.

Similarly to the unconstrained case, one implication of this result is that the optimality of a mixed-integer point can be verified in polynomial time, provided that the number of integer variables is a constant.

\begin{theorem}\label{mixedIntgerKKT}
Suppose that the problem \eqref{eq:constrained_MICO} is feasible. Let $g$ fulfill the  {\it mixed-integer Slater condition}.
A point $x^\star\in\Z^{n}\times\R^d$ is optimal for the mixed-integer constrained problem \eqref{chapter:CenterPoints::equ:mixedIntegerOpti} 
if and only if $g(x^\star)\le0$ and there exist $k\le 2^n$ points
$x_1=x^\star,x_2,\dots,x_k\in\Z^{n}\times\R^{d}$ and $k$ vectors $u_1,\dots,u_k\in\R_+^{m+1}$
with corresponding $h_{i,m+1}\in\partial f(x_i)$, and $h_{i,j}\in\partial g_j(x_i)$ for $j=1,\dots,m$ 
such that the following five conditions hold:
\begin{itemize}
\item[(a)] If $g(x_i)\le0$ then $f(x_i)\ge f({x^\star})$, $u_{i,m+1}>0$ and $u_{i,j}g_j(x_i)=0$ for $j=1,\dots,m$.
\item[(b)] {If $g(x_i)\nleq0$ then $u_{i,m+1}=0$ and $u_{i,j}=0$  for all $j\notin I_i:=\{1\leq j\leq m: g_j(x_i) = \max_{1\leq \ell\leq m}g_{\ell}(x_i)\}$.}
\item[(c)]  $1\le |\supp(u_i)| \le d+1$ for $i=1,\dots,k$.
\item[(d)] $\{x\in\R^{n+d} \l  \sum_{j=1}^{m+1}u_{i,j}h^\T_{i,j} (x-x_i) < 0 \text{ for all } i\}\cap(\Z^{n}\times\R^d)=\emptyset$.
\item[(e)]  $\sum_{j=1}^{m+1}u_{i,j}h_{i,j}\in\R^n\times\{0\}^d$ for $i=1,\dots,k$.
\end{itemize}
\end{theorem}

\begin{proof}
In order to prove the first implication, we assume that $x^\star$ is optimal.
Let $X^\star$ denote the set of all optimal solutions to \eqref{eq:constrained_MICO}. This set is not empty by assumption, and contains $x^\star$.
If there exists a point $x\in X^\star$ with $0\in\partial f(x)$, then the 
theorem follows directly from the purely continuous version of the KKT conditions described above; we can also take $k = 1$.
Next, assume there exists an $x\in X^\star\cap\intt(\conv(X^\star))$ and $h_x\in\partial f(x)$ such that $h_x\neq0$.
By convexity, $f$ must be constant on $\conv(X^\star)$, contradicting that $h_x\neq 0$.
This implies that if $X^\star\cap\intt(\conv(X^\star))\neq\emptyset$, then $0\in\partial f(x)$ for all $x\in X^\star$.
Hence, let us assume that $X^\star\cap\intt(\conv(X^\star))=\emptyset$ and that $0\notin\partial f(x)$ for all $x\in X^\star$.

For every $z\in\Z^n$ we  consider the following continuous convex subproblem, 
\begin{equation}\label{chapter:CenterPoints::equ:ContSubProb}
\min\limits_{y\in\R^d,\atop g((z,y))\le0} f((z,y)).
\end{equation}
We distinguish two cases.

\bigskip
(i) {\bf Problem \eqref{chapter:CenterPoints::equ:ContSubProb} is infeasible}. Let us define
$$y_z:=\argmin_{y\in \R^d}\max_{1\le i \le m}g_i((z,y)).$$
Let {$I_z:=\{1
\leq j\leq m \l g_j((z,y_z))=\max_{1\le i \le m}g_i((z,y_z))\}$}.
Since $y_z$ is an optimal solution to an unconstrained convex problem{,} there exists a vector $h_z\in\partial\max_{1\le i \le m}g_i((z,y_z))=\conv(\{\partial g_{i,z}((z,y_z))\l i\in I_z\})$
such that $h_z\in\R^n\times\{0\}^d$.
We can write $$h_z=\sum_{j=1}^{m}u_{z,j}h_{z,j}$$ with $u_{z,j}\ge0$ for $j\in I_z$, $u_{z,j} = 0$ for $j\notin I_z$, $\sum_{j\in I_z}u_{z,j} = 1$ and $h_{z,j}\in\partial g_j((z,y_z))$ for $j=1,\dots,m$. We also define $u_{z,m+1}:=0$.
From Caratheodory's Theorem it follows that we can choose {$u_z  = (u_{1,z},\ldots,u_{m,z},u_{z,m+1}) = (u_{1,z},\ldots,u_{m,z},0)$} such that $|\supp(u_{{z}})|\le d+1$. We verify Conditions (b), (c), and (e).

\bigskip

(ii) {\bf Problem \eqref{chapter:CenterPoints::equ:ContSubProb} is feasible}. We define
$$y_z:=\argmin_{y\in \R^d}\{f((z,y))\l g((z,y))\le0\}.$$
Since by our initial assumption $x^\star$ is optimal,  it follows that $f((z,y_z))\ge f(x^\star)$. From Slater's condition, we can apply the standard continuous KKT conditions. There exist a vector of multipliers $u_z\in\R^{m+1}_+$, a vector 
$h_{z,m+1}\in\partial f((z,y_z))$, and vectors $h_{z,j}\in\partial g_j((z,y_z))$ for $j=1,\dots,m$
such that 
$$u_{z,m+1}>0,\;  
u_{z,j}g_j((z,y_z))=0 {\text{ for } j=1,\dots,m}.$$
Condition (a) would thereby be verified whatever fiber minimizer we would take in our certificate.

Furthermore, the KKT conditions implies also that $\sum_{j=1}^{m+1}u_{z,j}h_{z,j}\in\R^n\times\{0\}^d$, which will lead to Condition (e). 

Note that $u_{z,m+1}>0$ implies that $|\supp(u)|\ge 1$.  
Caratheodory's Theorem implies that we can choose $u$ such that $|\supp(u)| \le d+1$, which will yield Condition (c).

Next we define
$$
h_z:=\sum_{j=1}^{m+1}u_{z,j}h_{z,j}\in\R^n\times\{0\}^d
$$
and the open half-space
\[
L_z:=\{ (z',y')\in\R^n \l h_z^\T ((z',y') - (z,y_z)) < 0\}.
\]
Since the last $d$ components of $h_z$ are null, $(z',y')$ belongs to $L_z$ if and only if the whole fiber containing $z'$ belongs to $L_z$. So, the fiber of $z$ does not belong to $L_z$ because $(z,y_z)\notin L_z$.

Suppose now that $(z',y')$ is a (continuous) feasible point that does not belong to $L_z$. Then $f((z',y'))\geq f((z,y_z))$. Indeed, we first have:
\[
0\geq\sum_{j=1}^mu_{z,j}g_j(z',y') = \sum_{j=1}^mu_{z,j}(g_j(z',y')-g_j(z,y_z))\geq \sum_{j=1}^mu_{z,j}h_{z,j}^\T ((z',y') - (z,y_z)),
\]
where we have used successively the nonnegativity of the multipliers $u_{z,j}$, the complementarity conditions, and the convexity of the functions $g_j$. Since $u_{z,m+1}>0$, we deduce that:
\[
0\leq h_{z,m+1}^\T ((z',y') - (z,y_z))\leq f((z',y')) - f((z,y_z))
\]
as announced.

Therefore, the intersection $L:=\cap_{z\in\Z^n}L_z$ is lattice-free and contains every feasible point $(z,y)$ for which $f((z,y))<f(x^\star)$. We have excluded at the beginning of this proof the situations where $L$ could be empty.

It follows from \cite{Doignon1973} that a sub-selection of $k\le2^n$ inequalities $h_{z_i}^\T y < h_{z_i}^\T z_i $, $i=1,\dots,k$ suffices to describe a lattice-free polyhedron containing $L$. By convexity, one of those $k$ points has to be a solution $x^\star$ to our mixed-integer problem. Then we obtain the desired certificate by defining $x_1=x^\star,x_2=(z_2,y_{z_2}),\dots,x_k=(z_k,y_{z_k})$. Note that we have written $h_i$ for $h_{z_i}$ in the theorem's statement.

All the conditions are now satisfied.

To prove the other direction, let $x_1,\ldots,x_k$ be the points in the certificate, and consider the open polyhedron: 
$$P:=\{x\in\R^{n+d} \l  \sum_{j=1}^{m+1}u_{i,j}h^\T_{i,j} (x-x_i) < 0 \text{ for all } i = 1,\ldots,m\},$$
with $h_{i,j}$ as defined in the statement of the Theorem.
We assume that Conditions (a) -- (e) are satisfied. In particular, $P$ is lattice-free.
Let $\bar x\in\Z^n\times\R^d$.
Then $\bar x$ must violate at least one inequality of $P$, say the $i$-th inequality, i.e., 
$$\sum_{j=1}^{m+1}u_{i,j}h^\T_{i,j} (\bar x-x_i) \geq 0.$$
Since $v:=\sum_{j=1}^{m+1}u_{i,j}h_{i,j}$ is a subdifferential of the convex function $\psi(x):=u_{i,m+1}f(x)+\sum_{j=1}^{m}u_{i,j}g_j(x)$ at $x_i$, the minimum of $\psi(x)$ over the half-space $\{x \mid v^\T (x-x_i) \geq 0\}$ is precisely $x_i$.

Assume $g(x_i)\nleq0$. Since $u_{i,m+1} = u_{i,j} = 0$ for all $j\notin I_i$, we get
\[
0\leq\sum_{j\in I_i}u_{i,j}h^\T_{i,j} (\bar x-x_i)\leq \sum_{j\in I_i}u_{i,j}(g_j(\bar x) - g_{\max}(x_i)),
\]
from which we deduce that $\bar x$ is not feasible.

Assume now that $g(x_i)\le0$. Then $u_{i,j}g_j(x_i) = 0$, so \begin{eqnarray*}
0\leq \sum_{j = 1}^{m+1}u_{i,j}h^\T_{i,j} (\bar x-x_i) &\leq &\sum_{j = 1}^{m}u_{i,j}(g_j(\bar x)-g_j(x_i)) + u_{i,m+1}(f(\bar x)-f(x_i))\\
&=&\sum_{j = 1}^{m}u_{i,j}g_j(\bar x) + u_{i,m+1}(f(\bar x)-f(x_i)). 
\end{eqnarray*}
So, if $\bar x$ is feasible, the sum of the $m$ first terms above is nonpositive, and $f(\bar x)\geq f(x_i)$. Thus the best point $x^\star$ among those in the certificate must be optimal.

\end{proof}

As an illustration of the above theorem, let us consider the mixed-integer Euclidean Projection problem:
\begin{equation}\label{eq:MIP_projection}
y^\star = \arg\min \{\norm x-y \norm_2 \mid g(x) \leq 0,  x \in \R^n\times\Z^d\},
\end{equation}
for a given point $y\in\R^{n+d}$.

The continuous version of this problem has a unique solution, say $y_{\textrm{cont}}$, which satisfies the so-called \emph{projection condition}:
\[
\text{for all feasible } x,\quad(y-y_{\textrm{cont}})^\T(x-y_{\textrm{cont}})\leq 0.
\]
Observe that this projection condition implies:
\[
\text{for all feasible } x,\quad\norm x-y \norm_2 \geq  \norm x-y_{\textrm{cont}} \norm_2.
\]
\begin{corollary}
{Assume that the feasible set fulfills the mixed-integer Slater's condition. The certificate $x_1,\ldots,x_k$ given for the problem \eqref{eq:MIP_projection} satisfies the projection property:
\[
\text{for all mixed-integer feasible } x\text{ there exists } 1\leq i\leq k\text{ for which }(y-x_i)^\T(x-x_i)\leq 0,
\]
and
\[
\text{for all mixed-integer feasible } x\text{ there exists } 1\leq i\leq k\text{ for which }\norm x-y \norm_2 \geq  \norm x-x_i \norm_2.
\]}
\end{corollary}
\begin{proof}
{
We denote by $S$ the feasible set $\{x\in\R^{n+1}\mid g(x)\leq 0\}$. Theorem \ref{mixedIntgerKKT} for $f(x):=||x-y||^2_2/2$ provides us with a certificate of points $x_1,\ldots,x_k$, its accompanying set of nonnegative vectors $u_i$ and subgradients $h_{i,j}$. Let $x\in S\cap(\Z^n\times\R^d)$. By Condition (d), the point $x$ violates at least one of the inequalities describing the open lattice-free polyhedron, say $\sum_{j=1}^{m+1}u_{i,j}h_{i,j}^\T(x-x_i)\geq 0$. By feasibility of $x$, we know that $u_{i,m+1}>0$. Thus:
\begin{eqnarray*}
0&\leq& \sum_{j=1}^mu_{i,j}h_{i,j}^\T(x-x_i) + u_{i,m+1}h_{i,m+1}^\T(x-x_i)\\
&\leq&\sum_{j=1}^mu_{i,j}(g_j(x)-g_j(x_i)) + u_{i,m+1}(x_i-y)^\T(x-x_i)\\
&=&\sum_{j=1}^mu_{i,j}g_j(x) + u_{i,m+1}(x_i-y)^\T(x-x_i)\\
&\leq& u_{i,m+1}(x_i-y)^\T(x-x_i).
\end{eqnarray*}
The second inequality comes readily from:
\[
||x-y||^2_2 = ||x-x_i||^2_2 + ||x_i-y||^2_2 + 2(x_i-y)^\T(x-x_i)\geq ||x-x_i||^2_2 + ||x_i-y||^2_2.
\]
}
\end{proof}

The following theorem characterizes another set of optimality conditions. We use here a larger set of points in the certificate, and therefore a more complex lattice-free polyhedron. Moreover, these points do not necessarily belong to the lattice. However, the Slater condition becomes much simpler to verify.

\begin{theorem}\label{StrongerOptimalityCertificate}
Assume the standard Slater's condition: there exists a point $s\in\R^{n+d}$ such that $g(s)<0$.
A point $x^\star\in\Z^{n}\times\R^d$ is optimal with respect to \eqref{chapter:CenterPoints::equ:mixedIntegerOpti} 
if and only if 
\begin{enumerate}
\item[(a)] $g(x^\star)\le0$,

\item[(b)] there exist $k+l\le 2^n(d+1)$ points
$x_1=x^\star,x_2,\dots,x_k$ and $y_1,\dots,y_l$ in $\R^{n+d}$ such that $f(x_i)\ge f(x^\star)$ for $i=1,\dots,k$ and $g(y_i)\le 0$ for $i=1,\dots,l$,

\item[(c)] there exist numbers $1\leq j_1,\ldots,j_l\leq m$ such that $g_{j_i}(y_i) = 0$ and there exist subgradients $h_{x_i}\in\partial f(x_i)$ and $h_{y_i}\in\partial g_{j_i}(y_i)$ such that 
\begin{align*}
P:=\{x\in\R^{n+d} \l & h_{x_i}^\T x < h_{x_i}^\T x_i \text{ for all } i=1,\dots,k, \\
& h_{y_i}^\T x\, \le  h_{y_i}^\T y_i \,\text{ for all } i=1,\dots,l\,\}
\end{align*}is lattice-free.
\end{enumerate}
\end{theorem}
\begin{proof}


Let $X^\star$ denote the set of all optimal solutions. We can assume that $X^\star$ is not empty, for otherwise the statement holds vacuously. In view of Slater's condition, the case where a point of $X^*$ coincides with a continuous optimum corresponds to the continuous KKT conditions \eqref{eq:KKT_continuous}, which imply the stated optimality certificate: if all the Lagrange multipliers are null, we simply take $k = 1$, $l = 0$, $x_1 = x^\star$ to ensure that the set $P$ is empty; if the Lagrange multipliers are not all null, we take $k = l = 1$, $x_i = y_1 = x^\star$, and, using the notation of \eqref{eq:KKT_continuous}, we take $h_{y_1}:=\sum_{j=1}^m\lambda_jh_{g_j}/\sum_{j=1}^m\lambda_j$ to get an empty $P$.

Therefore, we assume without loss of generality that no $x^\star\in X^\star$ has a nonnegative $\lambda\in\R^m$ that satisfies the KKT conditions \eqref{eq:KKT_continuous}. Using the argument of the proof of Theorem \ref{mixedIntgerKKT}, we can assume that $X^\star\cap\intt(\conv(X^\star))$ in empty.

Denote by $L$ the level set $L:=\{ x\in\R^{n+d} \l f(x)\le  f(x^\star) \}$, where $x^\star\in X^\star$,
and by $S_1,\ldots,S_m$ the sets  $S_j:=\{ x\in\R^{n+d} \l g_j(x)\le0 \}$, whose intersection $S$ is the feasible set.
Thus $X^\star = L\cap S\cap (\Z^n\times\R^d)$. 

In fact, our assumptions ensure that $\intt(L)\cap S\cap (\Z^n\times\R^d)$ is lattice-free. Suppose otherwise and take a point $x^\star\in X^\star$ that would be in the above set. Then there exists a closed  ball $B$ centered in $x^\star$ and contained in $L$. By
convexity, the maximum of $f$ on $B$ is attained on $\bd(B)$ and cannot exceed $f(x^\star)$. Hence the
function $f$ is constant on $B$, so that $f'(x^\star) = 0$. Taking $\lambda := 0$ shows that $x^\star$ satisfies the continuous KKT conditions, a contradiction.

Since $f$ and $g$ are continuous, $L$ and $S$ are closed. Additionally, the existence of a point $\hat x$ with $f(\hat x)<f(x^\star)$ --- because $0\notin \partial f(x^\star)$ --- allows us to describe $L$ as the intersection of half-spaces defined by its boundary points and their corresponding subdifferentials:
$$ L = \bigcap_{\substack{z\in\bd(L),\\ h \in\partial f(z)} } \{x\in\R^{n+d}  \l h^\T x \le h^\T z\}.$$
The interior of $L$ is easily seen to coincide with: $$ \bigcap_{\substack{z\in\bd(L),\\ h \in\partial f(z)} } \{x\in\R^{n+d}  \l h^\T x < h^\T z\}.$$
Similarly, in view of Slater's conditions, every set $S_j$ can be described as:
\[
S_j = \bigcap_{\substack{z\in\bd(S_j),\\ {h \in\partial g_j(z)}} } \{x\in\R^{n+d}  \l h^\T x \le h^\T z\}, 
\]
so that $S$ is:
\[
S = \bigcap_{j=1}^m\bigcap_{\substack{z\in\bd(S_j),\\ {h \in\partial g_j(z)}} } \{x\in\R^{n+d}  \l h^\T x \le h^\T z\} = \bigcap_{j=1}^m\bigcap_{\substack{z\in\bd(S_j)\cap S,\\ {h \in\partial g_j(z)}} } \{x\in\R^{n+d}  \l h^\T x \le h^\T z\}.
\]

It follows from \cite{AverkovWeismantel12} that a subset of $2^n(d+1)$ 
half-spaces suffice in order to guarantee that the corresponding intersection remains mixed-integer lattice-free. Without loss of generality, we can choose one of the supporting half-spaces in this description (see the construction in \cite{AverkovWeismantel12}); thus, we can take the inequality in the description of $L$ corresponding to $x^\star$ for $x_1$.

To prove the other direction, let $x_1,\ldots,x_k,y_1,\ldots,y_l$ be the points of the certificate and consider the lattice-free set $P$ given by the statement of the theorem. Let $x\in\Z^n\times\R^d$. Then $x$ violates one of the inequalities in the description of $P$. If that violated inequality is of the form $h_{x_i}^\T x < h_{x_i}^\T x_i$, for a point $x_i$ in $\bd(L)$, then we have $f(x)\geq f(x_i) + h_{x_i}^\T (x-x_i)\geq f(x_i) = f(x^\star)$ and $x$ cannot be better than $x^\star$. Otherwise the violated inequality is of the form $h_{y_i}^\T x \leq h_{y_i}^\T y_i$, for a point $y_i$ in $\bd(S_{j_i})$ and a subgradient $h_{y_i}\in\partial g_{j_i}(y_i)$. Since $g_{j_i}(y_i) = 0$, we have
\[
0<h_{y_i}^\T (x-y_i) = g_{j_i}(y_i) + h_{y_i,j}^\T (x-y_i))\leq g_{j_i}(x),
\]
so $x$ is not feasible and therefore not a solution to our problem.
\end{proof}


%
%

\section{The mixed-integer convex dual}

In the purely continuous setting, it is not too difficult to apply the KKT-theorem in order to show a duality theorem. Provided that the standard Slater condition holds, that all functions $f$ and $g_j,\; j=1,\ldots,m$ are continuous and convex, and that the primal and dual feasible sets are nonempty, one has
\begin{equation}\label{eq:continuous_convex_dual}
f^\star = \min\limits_{x\in\R^n }\{\;f(x)\;|\; g(x)\le0\} = \alpha^\star:=\max\limits_{\alpha, u\in\R_+^m }\{\alpha \mid \alpha \leq f(x) + u^Tg(x)\quad \forall x \in \R^n\}.
\end{equation}
In other words, any multiplier $ u \geq 0$ gives rise to an unconstrained convex optimization problem $\alpha = \min\{f(x) + u^Tg(x)\mid x \in \R^n\}$ whose optimum is a lower bound on the optimal primal value $f^\star$.
The naive extension of the continuous convex dual  would be to replace $\R^n$ in  \eqref{eq:continuous_convex_dual} everywhere  by $\Z^n \times \R^d$. This is not correct, though.

\begin{example}
Let $n  = m= 2$, $d = 0$, $f(x) = \frac{1}{2}||x-\mathbf{1}||^2_2$, $g_1(x) = x_1-\frac{1}{2}$, and $g_2(x) = x_2-\frac{1}{2}$, where $\mathbf{1}$ is the all-one vector. The mixed integer optimal point is $x^\star = 0$, so $f(x^\star) = 1$. However,
\[
f(x) + u^Tg(x) = \frac{1}{2}||x-\mathbf{1}||^2_2 + u_1x_1+u_2x_2 - \frac{1}{2}(u_1+u_2),
\]
which equals $\frac{1}{2} + \frac{u_1-u_2}{2}$ at $x = (1,0)^T$ and $\frac{1}{2} + \frac{u_2-u_1}{2}$ at $x = (0,1)^T$, showing that we only obtain a dual bound $\alpha^\star\leq \frac{1}{2}$.
\end{example}

Instead of using just one multiplier per constraint, our dual, as in our KKT Theorem \ref{mixedIntgerKKT} for mixed-integer problems, must use a selection of up to $2^n$ multipliers per constraint. To every mixed integer point is associated one multiplier of the selection. We denote the function that describes this association as $\pi:\Z^n\times\R^d \mapsto \{1,\dots,2^n\}$. 

Formally, the dual object is the pair of the function $\pi$ and the matrix $U\in\R^{2^n \times m}$ that pile up the multipliers. 
Here is a geometrical interpretation of that dual object. Consider a polyhedron $P=\{x\in\R^{n+d} \l Ax\le b\}$, where $A\in\R^{2^n\times (n+d)}$ and $b\in\R^{2^n}$ with a lattice-free interior that contains the continuous optimum and suppose that we are given vectors $U_i\in\R^m_+$ associated to each row of $Ax\le b$. We can use this polyhedron and these nonnegative vectors to generate a lower bound $\alpha$ for the original minimization problem \eqref{chapter:CenterPoints::equ:mixedIntegerOpti} as follows.
For each $i=1,\dots,2^n$, we consider the continuous convex problem
\begin{equation}\label{lowerBoundsFeasibility}
\min\limits_{x\in\R^{d+n}} \{U_ig(x)\mid A_i x \ge b_i\}.
\end{equation}
If the half-space $\{x\in\R^{d+n}\mid A_i x \ge b_i\}$ contains feasible points, the optimal value of \eqref{lowerBoundsFeasibility} is non-positive. Denote by $I_P$ the set of index $i$ of all those half-spaces containing feasible points. We can write:
\begin{equation*}
f(x^\star) \ge \alpha_{P,U}=\min_{i \in I_P} \big\{\min\limits_{x\in\R^{d+n}} f(x) + U_i g(x)\mid A_i x \ge b_i \big\}.
\end{equation*}
These lower bounds lead to the mixed integer duality result stated below.

\begin{theorem}
Let $f:\R^{n+d}\mapsto\R$ and $g:\R^{n+d}\mapsto\R^m$ be convex functions, and assume that the mixed-integer feasible set $\{x\in\Z^{n}\times\R^{d}\l g(x)\le0\}$ is non-empty, compact and contained in the domain of $f$.
Further, assume that $g$ fulfills the  {\it mixed-integer Slater condition}.
Then $\min\{\;f(x)\;|g(x)\le0, x\in\Z^n\times\R^d\}$ equals:
\begin{eqnarray*}
\max\limits_{\alpha\in\R\atop U\in\R_+^{2^n\times m}}\{\;\alpha &\l& \exists\; \pi: \Z^n\times\R^d \mapsto \{1,\dots,2^n\} \text{ satisfying:}\\
&& \forall
x\in\Z^n\times\R^d \text{ we have }
\alpha \le f(x) + U_{\pi(x)} g(x) \text{ or }1 \le U_{\pi(x)} g(x) \},
\end{eqnarray*}
where $U_i$ denotes the $i$-th row of $U$.
\end{theorem}

\begin{proof}
We call the minimization problem the primal problem and the maximization problem the dual problem.
The assumptions regarding $f$ and $g$ guarantee that there exists a feasible point $x^\star$ such that $f(x^\star)$ attains the primal optimum.
Then, for any $u\in\R_+^m$ we obtain using the point $x^\star$ as a condition on the optimal dual value,
$$\alpha \le f(x^\star)+u^\T g(x^\star)\le f(x^\star).$$
This bound on $\alpha$ guarantees that the optimal dual solution must be less or equal than the primal value.

To show the other direction we apply Theorem~\ref{mixedIntgerKKT} using the same notation.
Since $x^\star$ is optimal it follows from  Theorem~\ref{mixedIntgerKKT} that there exist $u_1,\dots,u_k$ fulfilling the conditions in Theorem~\ref{mixedIntgerKKT}.
If $u_{i,m+1}>0$, we define $U_{i,j}:=\frac{u_{i,j}}{u_{i,m+1}}$ for $j=1,\dots,m$.
Otherwise, $u_{i,m+1}=0$ and we define $U_{i,j}:=\frac{u_{i,j}}{\mu}$ for $j=1,\dots,m$, where $\mu:=\sum_{j=1}^{m}u_{i,j}g_j(x_i)$. 
Note that, if $k<2^n$ we can introduce artificial redundant rows $U_i=U_k$ for $i=k+1,\dots,2^n$.
Now, we define $\pi$ as follows: $\pi(x):=\min\{\; i \l \sum_{j=1}^{m+1}u_{i,j}h^\T_{i,j} (x-x_i) \ge 0 \}$. The lattice-freeness of the set $P$ in the statement of Theorem~\ref{mixedIntgerKKT} shows that this function is well-defined for every lattice point. If $x$ happens to be infeasible, then $\sum_{i = 1}^mu_{i,j}g_j(x_i)>0$, i.e., $U_{i} g(x)\geq 1$. If $x$ is feasible, then $u_{i,m+1}>0$ and
\[
f(x) + U_{i} g(x)  = f(x) + \sum_{i = 1}^m\frac{u_{i,j}}{u_{i,m+1}}(g_j(x)-g_j(x_i)) \geq f(x_i) + \sum_{i = 1}^{m+1}\frac{u_{i,j}}{u_{i,m+1}}h^\T_{i,j} (x-x_i) \geq f(x_i) \geq f(x^\star).
\]
We conclude that the primal and the dual solution  attain the same objective function value.
\end{proof}

It is straightforward to generalize the previous result slightly:  one may drop the assumptions about the convex functions $f:\R^{n+d}\mapsto\R$ and $g:\R^{n+d}\mapsto\R^m$ that ensure that the primal and dual problem are feasible and bounded.
This then forces us to replace the minimum and the maximum with the infimum and the supremum, respectively.

Let us finally  comment on the linear case. In this special situation we have that $f(x)=c^\T x$ and $g(x)=Ax-b$, with $c \in\Q^{n+d}$, $A\in\Q^{m\times(n+d)}$ and $b\in\Q^m$.
This special setting allows us to simplify the min-max relation and at the same time  highlight the connection to mixed-integer free polyhedra.
Let us  assume without loss of generality that the row vectors $U_1,\dots,U_k$ correspond to the first type of inequality in the duality statement, i.e. $\alpha \leq f(x) + U_{i} g(x)$, and the remaining $U_i$'s correspond to the second  type of inequality, $1 \le U_{i} g(x)$.
We define
\begin{equation*}
\begin{array}{rl}
P(\alpha,U):=\{x\in\R^{n+d} \; | & \alpha - U_i b > (c^\T - U_i A)x \text{ for } i=1,\dots,k \\
& 1 - U_i b > -U_i Ax \text{ for } i=k+1,\dots,m\}.
\end{array}
\end{equation*}
Then the duality statement in the linear mixed-integer situation can be recast as follows:
\begin{corollary}
With the notation introduced above one has
\end{corollary}
$$\min\limits_{x\in\Z^n\times\R^d}\{\;c^\T x\;| Ax\le b\}
=\max\limits_{\alpha\in\R\atop U\in\R_+^{2^n\times m}}\{\;\alpha \l P(\alpha,U) \cap (\Z^n\times\R^d) = \emptyset \}.$$

\bibliographystyle{plain}
\bibliography{lit}
\end{document}